\documentclass[a4paper,12pt]{article}

\usepackage[T2A]{fontenc}
\usepackage[cp1251]{inputenc}
\usepackage{amsfonts,amsmath,amsthm,amscd,amssymb,latexsym,amsbsy,pb-diagram,cite,caption}
\usepackage[dvips]{graphicx}
\usepackage[mathscr]{eucal}

\newtheorem{definition}{Definition}[section]
\newtheorem{theorem}[definition]{Theorem}
\newtheorem{lemma}[definition]{Lemma}
\newtheorem{proposition}[definition]{Proposition}
\newtheorem{corollary}[definition]{Corollary}

\theoremstyle{definition}
\newtheorem{remark}[definition]{Remark}
\newtheorem{example}[definition]{Example}

\numberwithin{equation}{section}

 \DeclareMathOperator{\card}{card}
\DeclareMathOperator{\diam}{diam}
\DeclareMathOperator{\spec}{Spec}\DeclareMathOperator{\dip}{DP}
\begin{document}

\begin{center}
{\Large\bf Diameter and diametrical pairs of points in ultrametric
spaces }
\end{center}

\bigskip
\begin{center}
{\bf D. Dordovskyi, O. Dovgoshey  and E. Petrov}
\end{center}

\begin{abstract}
Let $\mathcal{F}(X)$ be the set of finite nonempty subsets of a
set $X$. We have found the necessary and sufficient conditions under
which for a given function $\tau:\mathcal{F}(X)\rightarrow\mathbb{R}$
there is an ultrametric on $X$ such that
$\tau(A)=\diam A$ for every $A\in\mathcal{F}(X)$. For  finite
nondegenerate ultrametric spaces $(X,d)$ it is shown that  $X$
together with the subset of  diametrical pairs of points of $X$
forms a complete $k$-partite graph, $k\geqslant 2$, and, conversely,
every finite complete $k$-partite graph with $k\geqslant 2$ can be
obtained by this way. We use this result to characterize
the finite ultrametric spaces $(X,d)$ having the minimal
\mbox{$\card\{(x,y):d(x,y)=\diam  X,\mbox{ } x,y \in X \}$} for given $\card X$.
\end{abstract}



\section{Introduction}
\normalsize

A metric space $(X,d)$ is {\it ultrametric} if the following {\it
ultrametric inequality}
$$
d(x,y)=d(x,z)\vee d(z,y)
$$
holds for all $x,y,z\in X$, where we put
$$
d(x,z)\vee d(z,y)=\max\{d(x,z),d(z,y)\}.
$$

The ultrametric spaces and, in particular, the finite ultrametric
spaces play an important role for the modern physics, data analysis,
mathematical taxonomy and many others branches of the knowledge (see, for example, \cite{L2}).
In particular, the following fact seems to be fundamental in the theory of
hierarchical models:
\begin{theorem}[see \cite{CM}]\label{th1.1}
For every dendrogram $\mathcal{D}$ (= nested
family of partitions) over a finite set $X$ there is an ultrametric
$d$ on $X$ representing this $\mathcal{D}$.
\end{theorem}
We do not define more exactly the sense of the word ``representing'' in Theorem~\ref{th1.1}
but note that the ``dual form'' of this theorem will be given in Theorem~\ref{th3.11} below.

There exist also the significant relations between (finite) ultrametric
spaces and many mathematical constructions.
The well known example is the canonical representation of finite
ultrametric spaces by the weighted rooted trees. Note that this
representation can be generalized for the case of representation of
quasi-ultrametric finite spaces by the weighted digraphs, \cite{GV}.
It is also shown in \cite{GV} that finite ultrametric spaces can be
viewed as $k$-person positional games. An isomorphism of categories
of ultrametric spaces and complete, atomic, tree-like, real graduated
 lattices was proved in \cite{L1}.

Some basic results in the theory of (finite) ultrametric spaces were
obtained quite recently.
The isometry groups of finite ultrametric spaces are completely
characterized in \cite{S}.
A set of necessary and sufficient conditions under which the Cartesian
 products of metric spaces are ultrametric was found  in \cite{DM}. The
 counterpart of Gromov-Hausdorff metric in the ultrametric spaces was studied in
 \cite{Z}, \cite{Q}. The infinitesimal structure of pointed metric spaces
 having ultrametric ``tangent'' spaces  was described in \cite{DD}.

In the present paper we solve the following two problems.

Let $(X,\rho)$ be a  metric space and let $A$ be a
subset of $X$. Recall that the {\it diameter} of $A$ is the quantity
$$
\diam A=\sup\{\rho(x,y): x,y\in A\}.
$$
We shall say that a two-element subset $\{a,b\}$ of $A$ is a $\rho$-diametrical pair 
for $A$ if $\rho(a,b)=\diam A$. Write $\mathbb{R}:=(-\infty,\infty)$, $\mathbb{R}^+:=[0,\infty)$ and
$\mathcal{F}(Y)$ for the set of all nonempty finite subsets of a set $Y$.

(a) Under what conditions on a function
$\tau:\mathcal{F}(Y)\rightarrow\mathbb{R}$ there is an ultrametric
$\rho:Y\times Y\rightarrow\mathbb{R}^+$ such that
$$
\diam A=\tau(A)
$$
for every $A\in\mathcal{F}(Y)$?

(b) Characterize the structure of simple graphs $G=(V,E)$ for which there is an ultrametric
 $\rho$ defined on the vertex set $V$ such that the edge set $E$ coincides with
the set of $\rho$-diametrical pairs.

\section{Diameter in ultrametric spaces}

The following theorem describes the characteristic properties of
diameters of finite subsets in ultrametric spaces.

\begin{theorem}\label{t2.1}
Let $X$ be a nonempty set. The following two statements are
equivalent for every function
$\tau:\mathcal{F}(X)\rightarrow\mathbb{R}$.

\begin{itemize}
\item[$(i)$] For all $A,B,C\in\mathcal{F}(X)$ the function $\tau$
satisfies the conditions:
\begin{itemize}
\item[$(i_1)$] $(\tau(A)=0)\Leftrightarrow(\card A=1)$;
\item[$(i_2)$]
\begin{equation}\label{eq2.1}
\tau(A\cup B)\leqslant\tau(A\cup C)\vee\tau(C\cup B).
\end{equation}
\end{itemize}

\item[$(ii)$] There is an ultrametric $\rho:X\times
X\rightarrow\mathbb{R}^+$ such that
\begin{equation}\label{eq2.2}
\tau(A)=\diam A
\end{equation}
for every $A\in\mathcal{F}(X)$.
\end{itemize}
\end{theorem}

\begin{remark}
Conditions (i$_1$) and (i$_2$) are closely related to the similar
conditions from the paper of P.~Balk \cite{Balk}. In fact, Balk's paper
was a starting point in our consideration.
\end{remark}

\begin{remark}\label{r2.3}
Let (ii) hold. Properties (i$_1$) and (i$_2$) imply that the map
$\tau:\mathcal{F}(X)\rightarrow\mathbb{R}$ is isotonic and
nonnegative. Indeed, substituting of $A$ for $B$ into \eqref{eq2.1}
we see that
$$
\tau(A)\leqslant\tau(C)
$$
if $A\subseteq C$. Putting $A=\{c\}$ in the last inequality, where $c$
is an arbitrary element of $C$, and using (i$_1$) we obtain
$\tau(C)\geqslant 0$ for $C\in\mathcal{F}(X)$.
\end{remark}

\begin{proof}[\it Proof of Theorem \ref{t2.1}]
(ii)$\Rightarrow$(i). Let $\rho$ be an ultrametric on $X$ such that
$$
\tau(A)=\diam A,\quad A\in\mathcal{F}(X).
$$
Property (i$_1$) is evident. To prove (i$_2$) suppose that
$$
\underset{x,y\in A\cup B}{\max}\rho(x,y)=\underset{x,y\in
A}{\max}\rho(x,y).
$$
It implies
$$
\diam(A\cup B)=\underset{x,y\in
A}{\max}\rho(x,y)\leqslant\underset{x,z\in A\cup
C}{\max}\rho(x,z)
$$
$$
\leqslant\underset{x,z\in A\cup C}{\max}\rho(x,z)\vee \underset{y,z\in B\cup
C}{\max}\rho(y,z)=\diam(A\cup C)\vee\diam(C\cup B).
$$

Similarly we obtain \eqref{eq2.1} if
$\underset{x,y\in A\cup B}{\max}\rho(x,z)=\underset{x,y\in B}{\max}\rho(x,y)$.

Suppose now that $\underset{x,y\in A\cup
B}{\max}\rho(x,y)=\underset{x\in A,y\in
B}{\max}\rho(x,y)=\rho(a,b)$ where $a\in A,\,b\in B $. Then for
every $c\in C$ we have
$$
\diam(A\cup B)=\underset{x\in A,y\in B}{\max}\rho(x,y)= \rho(a,b)
\leqslant\rho(a,c)\vee\rho(b,c)
$$
$$
\leqslant
\underset{x\in A,z\in C}
{\max}\rho(x,z)\vee \underset{y\in B,z\in C}{\max}\rho(y,z)
\leqslant\diam(A\cup
C)\vee\diam(C\cup B).
$$
Property (i$_2$) follows.

(i)$\Rightarrow$(ii). Let $\tau:\mathcal{F}(X)\to\mathbb{R}$ be a
function satisfying  (i$_1$) and (i$_2$). Let us define a
function $\rho$ on $X\times X$ by the rule
$$
\rho(x,y):=\tau(\{x,y\}), \qquad x,y\in X
$$
where $\{x,y\}$ is the set with the elements $x$ and $y$. Property
(i$_1$) implies that $\rho(x,y)=0$ if and only if $x=y$. Moreover,
$\rho$ is nonnegative (see Remark \ref{r2.3}). The ultrametric
inequality $\rho(a,b)\leqslant\rho(a,c)\vee\rho(c,b)$
follows from \eqref{eq2.1} with $A=\{a\}$, $B=\{b\}$ and $C=\{c\}$.
Thus, $\rho$ is an ultrametric. It still remains to prove
\eqref{eq2.2} for  $A\in\mathcal{F}(X)$ with $\card A \geq 3$. Let $A=\{x_1,\ldots,
x_n\}$. Then for $j=1,\ldots,n$ we obtain
$$
\tau(A)=\tau(\{x_1,\ldots,x_{n-1}\}\cup\{x_n\})\leqslant\tau(\{x_1,\ldots,x_{n-1}\}\cup\{x_j\})\vee\tau(\{x_n,x_j\})
$$
$$
=\tau(\{x_1,\ldots,x_{n-1}\})\vee\rho(x_n,x_j)\leqslant\tau(\{x_1,\ldots,x_{n-1}\})
\vee(\underset{1\leqslant j\leqslant n}{\max}\rho(x_n,x_j)).
$$
The repetition of this procedure gives
$$
\tau(A)\leqslant\underset{1\leqslant i\leqslant j\leqslant n
}{\bigvee}\rho(x_i,x_j)=\diam A.
$$
As was shown in Remark \ref{r2.3} $\tau$ is an isotonic map.
Consequently, if $x_i,x_j$ is a diametrical pair of points of $A$, then
$$
\tau(A)\geqslant\tau(\{x_i,x_j\})=\rho(x_i,x_j)=\diam A.
$$
Thus, \eqref{eq2.1} holds for every $A\in\mathcal{F}(X)$.
\end{proof}

\begin{remark}\label{r2.4}
The second part of Theorem \ref{t2.1} shows, in particular, that for
every function $\tau:\mathcal{F}(X)\to\mathbb{R}$ satisfying
(i$_1$)--(ii$_2$) there is a {\it unique} ultrametric $\rho:X\times
X\to\mathbb{R}^+$ such that $\rho(x,y)=\tau(\{x,y\})$ for all $x,y\in X$.
\end{remark}

Theorem \ref{t2.1} can be generalized to the case of functions $\tau$ defined also on
infinite subsets of $X$.

Let $\mathfrak{Y}(X)$ be the set of nonempty subsets of $X$.

\begin{theorem}\label{t2.5}
The following two statements are equivalent for every function
$\tau:\mathfrak{Y}(X)\to\mathbb{R}$.
\begin{itemize}
\item[$(i)$] For all $A,B,C\in\mathfrak{Y}(X)$ the function $\tau$
satisfies conditions (i$_1$) and (i$_2$) from Theorem 2.1 and
the following condition
\begin{itemize}
\item[$(i_3)$]
\begin{equation}\label{eq2.3}
\tau(A)\leqslant\sup\{\tau(B): B\in\mathcal{F}(A) \}.
\end{equation}
\end{itemize}

\item[$(ii)$] There is an ultrametric $\rho:X\times
X\to\mathbb{R}^+$ such that $\tau(A)=\diam A$ for every
$A\in\mathfrak{Y}(X)$.
\end{itemize}
\end{theorem}

This theorem can be obtained as a consequence of Theorem \ref{t2.1}.
Indeed, since $\mathcal{F}(X)\subseteq\mathfrak{Y}(X)$, condition
(i) of Theorem \ref{t2.5} implies condition (i) of Theorem
\ref{t2.1}. Consequently, the restriction $\tau|_{\mathcal{F}(X)}$
is a diameter of finite sets for some ultrametric $\rho:X\times
X\to\mathbb{R}^+$. If $\card A=\infty$, then the equality
$\tau(A)=\diam A$  follows from the same equality with
$\card A<\infty$ by application of condition (i$_3$). The details of the
proof we leave to the reader.

\begin{example}\label{ex2.6}
Let $X$ be an infinite set. Define a function
$\tau:\mathfrak{Y}(X)\to\mathbb{R}$ as
$$
\tau(A)=
\begin{cases}
0 \quad\text{if } \card A=1, \\
1 \quad\text{if } 1<\card A<\infty, \\
2 \quad\text{if } \card A=\infty.
\end{cases}
$$
It is easy to prove that condition (i$_3$) does not hold but (i$_1$)
and (i$_2$) take place for this $\tau$. Thus, (i$_1$) and (i$_2$)
do not imply (i$_3$).
\end{example}

The key property of ultrametric spaces is that their balls are
either disjoint or comparable w.r.t. the set-theoretic inclusion.
Let us extend this property to the ``balls'' in $\mathcal{F}(X)$. Let
$\tau$ be a real-valued function on $\mathcal{F}(X)$ satisfying
(i$_1$) and (i$_2$). For $r>0$ and $A\in\mathcal{F}(X)$ define a
``closed'' ball
\begin{equation}\label{eq2.3**}
B_r(A):=\{C\in\mathcal{F}(X) : \tau(A\cup C)\leqslant r \}
\end{equation}
with a center $A$ and radius $r$.

The following properties are evident:

\begin{itemize}
\item[$\bullet$] $B_r(A)=\varnothing$ if and only if $r<\tau(A)$;

\item[$\bullet$] Every nonempty subset $C$ of the set $A$ belongs to the ball $B_r(A)$ if this ball is nonempty;

\item[$\bullet$] $B_{r_1}(A)\supseteq B_{r_2}(A)$ if $r_1\geqslant r_2$.

\end{itemize}

\begin{proposition}\label{prop2.7}
Let $X$ be a nonempty set and $\tau:\mathcal{F}(X)\to\mathbb{R}$ be
a function satisfying condition (i) of Theorem \ref{t2.1}. Then for every ball
$B_r(A)$ and every $C \in \mathcal{F}(X)$ we have
\begin{equation}\label{eq2.5*}
B_r(C)=B_r(A)
\end{equation}
if $C\in B_r(A)$, furthermore, for each pair $B_{r_1}(A_1)$,
 $B_{r_2}(A_2)$ with $r_1\geqslant r_2$ we have either
\begin{equation}\label{eq2.4}
B_{r_1}(A_1)\cap B_{r_2}(A_2)=\varnothing \mbox{ or } B_{r_2}(A_2)\subseteq B_{r_1}(A_1).
\end{equation}

\end{proposition}

\begin{proof}
Let $C$ belong to $B_r(A)$. To prove the inclusion
\begin{equation}\label{eq2.6}
B_{r}(C)\subseteq B_{r}(A)
\end{equation}
suppose that $Y\in B_r(C)$, i.e. $\tau(Y\cup C)\leqslant r$. Since
$C\in B_r(A)$, we have $\tau(A\cup C)\leqslant r$. Using (i$_2$) we obtain
$\tau(A\cup Y)\leqslant\tau(A\cup C)\vee\tau(C\cup Y)$.
Consequently, $\tau(A\cup Y)\leqslant r$ holds, so \eqref{eq2.6} follows.
 To prove \eqref{eq2.5*} it remains to show that
\begin{equation}\label{eq2.6*}
B_{r}(A)\subseteq B_{r}(C).
\end{equation}
For every $Z\in B_r(A)$ we have the inequality
\begin{equation}\label{eq2.7}
\tau(A\cup Z)\leqslant r.
\end{equation}
Since $C\in B_r(A)$, we also obtain
\begin{equation}\label{eq2.8}
\tau(A\cup C)\leqslant r.
\end{equation}
Inequalities \eqref{eq2.7}, \eqref{eq2.8} and condition (i$_1$) imply
$$
\tau(C\cup Z)\leqslant\tau(C\cup A)\vee\tau(A\cup Z)\leqslant r\vee
r=r.
$$
Thus $Z\in B_r(C)$, that implies \eqref{eq2.6*}.

Suppose now that $B_{r_1}(A_1)$ and $B_{r_2}(A_2)$ be some balls
such that $r_1\geqslant r_2$. If $B_{r_1}(A_1)\cap B_{r_2}(A_2)\neq\varnothing$,
 then there is $C\in B_{r_1}(A_1)\cap B_{r_2}(A_2)$. As we already know,
the equalities $B_{r_1}(A_1)=B_{r_1}(C)$ and $B_{r_2}(A_2)=B_{r_2}(C)$
hold. Consequently, the inequality $r_1\geqslant r_2$ implies
\eqref{eq2.4}.
\end{proof}

Using Proposition~\ref{prop2.7} and the characteristic property of bases
for topological spaces \cite[p. 21]{En} we see that the family
$$
\textbf{B}(X):=\{B_r(A):A \in \mathcal{F}(X), r\geqslant \tau (A)\}
$$
is a base for some topology on $\mathcal{F}(X)$. It is well known that
the family of closed balls in each ultrametric space forms a clopen base for
the space, but, in general, it is not the case for the base $\textbf{B}(X)$.

Indeed, suppose that $a,b,c,d$ are some points of $X$ such that
\begin{equation}\label{eq2.12}
\tau (\{a,b\})<\tau(\{c,d\})=\sup\{\tau(A):A\in\mathcal{F}(X)\}
\end{equation}
and the ball
$$
B_{r_1}(\{a,b\}),\quad r_1:=\tau(\{a,b\})
$$
is closed in the topology generated by the base $\textbf{B}(X)$.
The inequality in~(\ref{eq2.12}) and~(\ref{eq2.3**}) imply
$$
\{c,d\}\in \mathcal{F}(X)\backslash B_{r_1}(\{a,b\}).
$$
By the supposition the set $\mathcal{F}(X)\backslash B_{r_1}(\{a,b\})$ is open. Since
$\textbf{B}(X)$ is a base, there is $B_{r_2}(A)\in \textbf{B}(X)$ such that
\begin{equation}\label{eq2.13}
\{c,d\}\in B_{r_2}(A)\subseteq\mathcal{F}(X)\backslash B_{r_1}(\{a,b\}).
\end{equation}
By Proposition~\ref{prop2.7}, we have the equality
$B_{r_2}(A)=B_{r_2}(\{c,d\})$.
The smallest $r\in\mathbb{R^+}$ for which $\{c,d\}\in B_r(\{c,d\})$ equals $\tau(\{c,d\})$.
Consequently, the inequality $r_2\geqslant \tau(\{c,d\})$ holds.
The last inequality and the inclusion in~(\ref{eq2.13}) show that
\begin{equation}\label{eq2.14}
\{a,b\} \notin B_{r_3}(\{c,d\})
\end{equation}
where $r_3:=\tau(\{c,d\})$. The equality in~(\ref{eq2.12}) and~(\ref{eq2.3**}) imply
$A \in B_{r_3}(\{c,d\})$
for every $A\in \mathcal{F}(X)$, contrary to~(\ref{eq2.14}).

\section{Diametrical pairs of points in ultrametric spaces}

To formulate the main results of the present section we recall the basic
definitions from the graph theory  (see \cite{BM} for all terms that
will be used without explicit formulations).

A graph $G$ is an ordered pair $(V,E)$ consisting of a set $V=V(G)$ of
{\it vertices} and a set $E=E(G)$ of {\it edges}. In this paper we consider
only {\it simple} graphs, so each edge is an unordered pair of distinct vertices
and $E(G)$ is a set of two-element subsets of $V(G)$.

\begin{definition}\label{def3.1}
Let $G$ be a simple graph and let $k$ be a cardinal number. The graph $G$ is
$k$-partite if the vertex set $V(G)$ can be partitioned into $k$ nonvoid disjoint
subsets, or parts, in such a way that no edge has both ends in the same part. A $k$-partite
graph is complete if any two vertices in different parts are adjacent.
\end{definition}

We shall say that $G$ is a {\it complete multipartite graph} if there is
 $k \geqslant 1$ such that $G$ is complete $k$-partite, cf. \cite[p. 14]{Di}.

\begin{remark}\label{rem3.2}
It is easy to see that  1-partite graph $G$ is empty in the sense that no two
vertices are adjacent.
\end{remark}

To simplify the formulations  we consider the cases of finite and
infinite ultrametric spaces separately.
\begin{theorem}\label{th3.3}
Let $(X,\rho)$ be an ultrametric space with $\card X = \infty$ and let $G$ be a graph such
that $V(G)=X$ and
\begin{equation}\label{eq3.1}
(\{u,v\}\in E(G))\Leftrightarrow (\{u,v\} \mbox{ is a }\rho\mbox{-}\mbox{diametrical pair
of points for } X).
\end{equation}
 Then $G$ is complete multipartite graph.
Conversely, if $G=(V,E)$ is a complete multipartite graph  with $\card V=\infty$, then
 there is an ultrametric $\rho: V\times V \rightarrow \mathbb{R}^+$ such that
 $E$ is the set of $\rho$-diametrical pairs of points for $V$.
\end{theorem}

\begin{theorem}\label{th3.4}
Let $(X,\rho)$ be an ultrametric space with $2\leqslant \card X < \infty$ and let $G$ be a graph such
that $V(G)=X$ and condition~(\ref{eq3.1}) holds. Then $G$ is a complete $k$-partite graph with $k\geqslant 2$.
Conversely, if $G=(V,E)$ is a finite ($\card V < \infty$) complete multipartite graph  with $k \geqslant 2$, then
 there is an ultrametric $\rho: V\times V \rightarrow \mathbb{R}^+$ such that
 $E$ is the set of $\rho$-diametrical pairs of points for $V$.
\end{theorem}

The proof of these theorems based on the following three lemmas.
\begin{lemma}\label{lem3.5}
Let $(X,\rho)$ be an ultrametric space with $\card X \geqslant 2$. Let
us define the binary relation $\equiv$ on $X$ by the rule
\begin{equation}\label{eq3.2}
(x \equiv y) \Leftrightarrow (\rho(x,y)<\diam X).
\end{equation}
Then $\equiv$ is an equivalence relation on $X$.
\end{lemma}
\begin{proof}
It is evident that $\equiv$ is symmetric and reflexive. Suppose now that we have $x\equiv y$ and
$y\equiv z$ for $x,y,z \in X$. Then, by the ultrametric inequality, we obtain.
$$
\rho(x,z)\leqslant \rho(x,y)\vee \rho(y,z)<(\diam X)\vee (\diam X) =  \diam X.
$$
Hence, $\equiv$ is transitive. Thus, $\equiv$ is an equivalence relation.
\end{proof}
Let us recall that the {\it complement} $\overline{G}$ of a graph $G=(V,E)$ is a
graph whose vertex set is $V$ and whose edges are the pairs of distinct
nonadjacent vertices of $G$.
\begin{lemma}\label{lem3.6}
Let $G=(V,E)$ be a graph. The following properties are equivalent:
\begin{itemize}
\item [(i)] $G$ is a complete multipartite graph;
\item [(ii)] Every connected component of the complement $\overline{G}$ is a complete graph.
\end{itemize}
\end{lemma}
\begin{proof}
Suppose $G$ is a complete multipartite graph. Let $\{Y_i:i\in \mathcal{I}\}$ be a partition
of the vertex set $V$ with an index set $\mathcal{I}$ such that the properties described in
Definition \ref{def3.1} hold. Let $u,v \in V$, $u\neq v$. Then $u$ and $v$ are adjacent 
 in $\overline{G}$ if they belong to the same part $Y_i$ and nonadjacent in the opposite case.
Hence, for every $i \in \mathcal{I}$ the subgraph of $\overline{G}$ induced by $Y_i$ is a connected
component of $\overline{G}$ and this component is a complete graph. Implication
(i)$\Rightarrow$(ii) follows. The proof of implication (ii)$\Rightarrow$(i) is also
straightforward and we omit this here.
\end{proof}
\begin{lemma}\label{lem3.7}
Let $X$ be a nonempty set and let $\equiv$ be an equivalence relation on $X$.
If $\card X =\infty$, then there is an ultrametric $\rho:X\times X \rightarrow \mathbb{R}^+$
such that~(\ref{eq3.2}) holds. For finite $X$ with $\card X \geqslant 2$, an ultrametric $\rho:X\times X \rightarrow \mathbb{R}^+$
satisfying~(\ref{eq3.2})  exists if and only if the relation $\equiv$ is different
from the universal relation $\{(x,y):x,y\in X\}$.
\end{lemma}
\begin{proof}
Consider first the case where $\card X=\infty$. If $\equiv$ is the universal relation on $X$, then
every ultrametric $\rho:X\times X\rightarrow \mathbb{R}^+$ satisfying the inequality
\begin{equation}\label{eq3.3}
\rho(x,y)<\diam X
\end{equation}
for all $x,y \in X$ satisfies also~(\ref{eq3.2}). Suppose that $\equiv$ is different from the universal
relation, then the partition of the set $X$ on the equivalence classes
$$\{y:y\equiv x\}, \quad x\in X$$
contains more than one class. Let us define a function $\rho$ on $X\times X$ as

\begin{equation}\label{eq3.4}
\rho(x,y):=\left\{
\begin{array}{l}
\mbox{ }0\quad \mbox{if } x=y,\\\
\frac{1}{2}\quad\mbox{if } x\neq y\mbox{ and } x\equiv y,\\\
1\quad\mbox{if }x\not\equiv y.\\
\end{array}
\right.
\end{equation}
We claim that $\rho$ is an ultrametric on $X$ and~(\ref{eq3.2}) holds for this $\rho$.
To prove that $\rho$ is an ultrametric it is suffice to show that
\begin{equation}\label{eq3.5}
\rho(x,z)\leqslant\rho(x,y)\vee\rho(y,z)
\end{equation}
for pairwise distinct $x,y,z \in X$. The proof is very simple. Since $x,y$ and $z$ are pairwise
distinct, inequality~(\ref{eq3.5}) does not hold if and only if $\rho(x,z)=1$ and 
$\rho(x,y)=\rho(y,z)=\frac{1}{2}$. Hence, by~(\ref{eq3.4}),  we obtain
$$
x \not\equiv z, \quad x\equiv y\mbox{ and } y\equiv z,
$$
contrary to the transitivity of $\equiv$.
Suppose now that $\card X <\infty$. If there is an ultrametric $\rho:X\times X\rightarrow \mathbb{R}^+$
such that~(\ref{eq3.2}) holds, then using the finiteness of $X$ we can find $x,y \in X$ such that
$$
\rho(x,y)=\diam X.
$$
Consequently, $\equiv$ is different from the universal relation. The remaining
part of the proof can be obtained by a repetition of the arguments given above.
\end{proof}
\begin{remark}\label{rem3.7*}
In the proof of Lemma~\ref{lem3.7} we use the existence of an ultrametric
$\rho$ satisfying~(\ref{eq3.3}) for all $x,y \in X$. The ``explicit'' construction
of such type metric for arbitrary $\card X$ we shall give in Example~\ref{ex3.10}.
\end{remark}

\begin{proof}[Proof of Theorem~\ref{th3.3}]
Let $(X,\rho)$ be an ultrametric space, $\card X =\infty$ and let $G$  be a graph
with $V(G)=X$ and $E(G)$ defined by~(\ref{eq3.1}). By Lemma~\ref{lem3.6}, $G$ is complete
multipartite if  every connected component of $\overline{G}$ is complete. If
$x$ and $y$ are distinct points, then
$$
(\{x,y\}\in E(\overline{G}))
\Leftrightarrow
 (\{x,y\}\notin E(G))
 \Leftrightarrow
(\rho(x,y)<\diam X)
 \Leftrightarrow
(x\equiv y),
$$
(see~(\ref{eq3.2})).
By Lemma~\ref{lem3.5}, the relation $\equiv$ is an equivalence on $X$.
The equivalence classes induced by $\equiv$ form the set of the connected
components of $\overline{G}$. All these components are complete because $\equiv$ is
transitive.

Conversely, if $G=(V,E)$    is a complete multipartite graph with $\card V =\infty$, then,
by Lemma~\ref{lem3.6},  every connected component of $\overline{G}$ is a complete graph.
Let us define the relation $\equiv$ on $V$ as
$$
(x\equiv y)\Leftrightarrow (x \mbox{ and } y \mbox{ belong to the same connected component of } \overline{G}).
$$
Since every component is complete, the relation $\equiv$ is an equivalence relation
on $V$. By Lemma~\ref{lem3.7}, there is an ultrametric
$\rho:V \times V\rightarrow \mathbb{R}^+$ such that
$$
(x\equiv y)\Leftrightarrow (\rho(x,y)<\diam V)
$$
for $x,y \in V$. Consequently, we obtain
$$
(x\not\equiv y)\Leftrightarrow(\{x,y\}\in E(\overline{\overline{G}}))
\Leftrightarrow(\{x,y\}\in E(G))
\Leftrightarrow (\rho(x,y)=\diam V).
$$
Thus, $E(G)$ is the set of $\rho$-diametrical pairs of $(V,\rho)$.
\end{proof}
Theorem~\ref{th3.4} can be proved in much the same way as Theorem~\ref{th3.3}.
\begin{remark}\label{rem3.8}
The second part of Lemma~\ref{lem3.7} is, in fact, a special particular case of Theorem~\ref{th1.1}.
Lemma~\ref{lem3.6} provides a transition of results from complete multipartite graphs to equivalence
relations and back. 
\end{remark}

Let $(X,d)$ be a nonempty metric space. Following  paper~\cite{DLPS}
we shall say that the set
$$
\spec(X,d)=\{d(x,y):x,y\in X\}
$$
is the {\it spectrum} of $(X,d)$.

Let $A$ be a subset of $\mathbb{R}^+$. If $0 \in A$, then there is an ultrametric
space $(X,\rho)$ such that $\spec (X,\rho)=A$.
This result was proved in \cite{DLPS} by the following elegant method. Define
$d: A \times A\rightarrow \mathbb{R}^+$, setting
\begin{equation}\label{eq3.6}
d(x,y)=
\left\{
\begin{array}{l}
x\vee y \quad \mbox{if } x \neq y, \\\
0 \quad\mbox{otherwise. }\\
\end{array}
\right.
\end{equation}
Then $(A,d)$ is an ultrametric space and $\spec (A,d)=A$.

Using ultrametric~(\ref{eq3.6}) we can simply fill the gap in the proof of Lemma~\ref{lem3.7}
mentioned in Remark~\ref{rem3.7*}.
\begin{example}\label{ex3.10}
Let $X$ be an infinite set, $A_c$ a countable infinite subset of $X$, $A$ a countable
infinite subset of $\mathbb{R}^+$, $0\in A$,
\begin{equation}\label{eq3.7}
\sup A \notin A
\end{equation}
and let $f:A_c\rightarrow A$  be a bijection. For an arbitrary $\varepsilon>0$
define a function $\rho:X\times X\rightarrow \mathbb{R}^+$ as
\begin{equation}\label{eq3.8}
\rho(x,y):=\left\{
\begin{array}{l}
0 \quad\mbox{if } x=y, \\\
\varepsilon\quad \mbox{if } x\neq y \mbox{ and } x,y\in X\backslash A_c, \\\
\varepsilon+d(f(x),f(y))\quad \mbox{if } x\neq y \mbox{ and } x,y\in  A_c, \\\
\varepsilon+f(x)\quad \mbox{if } x\in A_c \mbox{ and } y\in  X\backslash A_c, \\\
\varepsilon+f(y)\quad \mbox{if } x\in X\backslash A_c \mbox{ and } y\in  A_c \\
\end{array}
\right.
\end{equation}
where $d$ is defined by~(\ref{eq3.6}). It is simple to see that $\rho$ is symmetric
and nonnegative. To prove the ultrametric inequality it is sufficient to show that
\begin{equation}\label{eq3.9}
\rho(x,y)\leqslant\rho(x,z)\vee\rho(z,y)
\end{equation}
for $x\in A_c$, $y\in X\backslash A_c$ and $x\neq z\neq y$.
Using these relations,~(\ref{eq3.6}) and~(\ref{eq3.8}) we obtain
$$
\rho(x,y)=\varepsilon+f(x), \quad \rho(x,z)\geqslant \varepsilon+f(x) \quad \mbox{and} \quad \rho(z,y)\geqslant \varepsilon,
$$
so~(\ref{eq3.9}) follows. Hence, $\rho$ is an ultrametric.
\end{example}

 Note that~(\ref{eq3.8}) and~(\ref{eq3.7}) imply~(\ref{eq3.3}). Lemma~\ref{lem3.7} is
 completely proved now.

Using the idea of spectrum of ultrametric spaces we can give the following ``dual form'' of
Theorem~\ref{th1.1}.
\begin{theorem}\label{th3.11}
Let $X$ be  a finite set of points, $\card X \geqslant 2$, $n$ be an integer number,
 $1\leqslant n \leqslant \card X -1$, and let $\{G_1,...,G_n\}$ be
a  set of complete multipartite graphs with $V(G_1)=...=V(G_n)=X$ satisfying the following conditions:
\begin{itemize}
  \item [($i_1$)] $G_1$ is complete and $G_n$ is nonempty in the sense that $\{x,y\}\in E(G_1)$
  for all distinct $x,y \in X$ and $E(G_n)\neq\varnothing$;
  \item [($i_2$)] $E(G_{i+1})$ is a proper subset of $E(G_i)$ for $i=1,...,n-1$.
\end{itemize}
Then there is an ultrametric space $(X,\rho)$ with 
\begin{equation}\label{eq3.10}
\spec(X,\rho)=\{a_0,a_1,...,a_n\}, 0=a_0<a_1<...<a_n,
\end{equation}
such that for $x,y \in X$
\begin{equation}\label{eq3.11}
(\{x,y\} \in E(G_i))\Leftrightarrow (\rho(x,y)\geqslant a_i), \quad i=1,...,n.
\end{equation}
Conversely, if $(X,\rho)$ is an ultrametric
space with spectrum~(\ref{eq3.10}) and if $\{G_1,...,G_n\}$ is a set of graphs such
that $V(G_1)=...=V(G_n)=X$ and~(\ref{eq3.11}) holds,  then all $G_i$ are complete
multipartite graphs meeting conditions ($i_1$) and ($i_2$).
\end{theorem}
This theorem lies beyond our considerations of diameters in ultrametric spaces
so we omit the proof here.

\section{Finite ultrametric spaces with minimal sets of diametrical pairs of points}
Let $(X,d)$ be a nonempty finite ultrametric space. Let us denote by $\dip(X)$
the set of all {\bf ordered} pairs $(x,y)$, $x,y\in X$, for which $d(x,y)=\diam X$, i.e.
$$
\dip(X)=\{(x,y)\in X\times X: d(x,y)=\diam X\}.
$$

It is almost evident that the set $\dip(X)$ cannot be a ``small'' subset of  $X\times X$.
In this section we describe finite ultrametric spaces $(X,d)$ for which $\card \dip(X)$
is smallest among all ultrametric spaces having the same number of elements.

We start from the following simple example.
\begin{example}\label{ex4.1}
Let $(X,d)$ be a nonempty ultrametric space, let $a$ be a point such that $a\notin X$
and let $t$ be a real number for which $\diam X<t$. Write $X_a:=X\cup\{a\}$ and define
\begin{equation*}%\label{eq4.1}
d_{a}(x,y):=\left\{
\begin{array}{l}
d(x,y)\quad \mbox{if } x,y \in X, \\\
0\quad \mbox{if } x=y=a, \\\
t\quad \mbox{if } x\in X, y=a   \mbox{ or } y\in X, x=a.\\
\end{array}
\right.
\end{equation*}
It is easy to see that $(X_a,d_a)$ is an ultrametric space and
$$
(d_a(x,y)=\diam X_a)\Leftrightarrow ((x=a \mbox{ and }y \in X) \mbox{ or } (y=a \mbox{ and }x \in X))
$$
and $\card DP(X_a)=2\card X$.
\end{example}

\begin{theorem}\label{th4.2}
Let $(X,\rho)$ be a finite ultrametric space with $\card X\geqslant 2$. Then the inequality
\begin{equation}\label{eq4.2}
\card \dip(X)\geqslant 2(\card X-1)
\end{equation}
holds. Equality in~(\ref{eq4.2}) is attained if and only if there is an ultrametric
space $(Y,d)$ such that $(X,\rho)$ is isometric to  $(Y_a,d_a)$.
\end{theorem}
We divided the proof of Theorem~\ref{th4.2} into next two lemmas.

\begin{lemma}\label{lem4.4}
Let $(X,\rho)$ be a finite ultrametric space with $\card X\geqslant 2$. Then inequality~(\ref{eq4.2})
holds.
\end{lemma}
\begin{proof}
Let $x_1$ and $x_2$ be fixed points of $X$ such that $\rho(x_1,x_2)=\diam X$.
The ultrametric inequality implies that $\rho(x_1,y)=\diam X$ or $\rho(x_2,y)=\diam X$
for every $y\in X\backslash\{x_1,x_2\}$. Consequently, the number of diametrical
pairs of points is greater or equal than $1+\card (X\backslash \{x_1,x_2\})=\card X-1$. Inequality~(\ref{eq4.2}) follows.
\end{proof}

\begin{lemma}\label{lem4.3}
Let $(X,\rho)$ be a finite ultrametric space with $\card X\geqslant 2$. If
\begin{equation}\label{eq4.6}
\card \dip(X) = 2(\card X-1),
\end{equation}
then there is  $x_0 \in X$ such that
\begin{equation}\label{eq4.7}
\diam X=\rho(x_0,x)
\end{equation}
for every $x\in X\backslash \{x_0\}$ and
$\diam X>\rho(x,y)$ for all $x,y \in X\backslash \{x_0\}$.
\end{lemma}
\begin{proof}
Suppose that~(\ref{eq4.6}) holds. By Theorem~\ref{th3.4} the set $X$ together
with the set of $\rho$-diametrical pairs of points  form a complete  $k$-partite
graph $G$ with $k \geqslant 2$. We must show that $G$ is bipartite ($k=2$) with
a bipartition $(X_1,X_2)$,
\begin{equation}\label{eq4.7*}
X_1\cup X_2=X, \quad X_1\cap X_2=\varnothing
\end{equation}
such that $\card X_1=1$ or $\card X_2=1$.
Let $X_1,...,X_k$ be a partition of $X=V(G)$. Write
$$
m_i=\card X_i,\quad i=1,...,k
$$
and $n=\card X$. If $k\geqslant 3$, then using Lemma~\ref{lem4.4}
and the equality $\sum\limits_{i=2}^{k}m_i=n-m_1$ we obtain
$$
\card DP(X)=\sum\limits_{i=1}^{k}m_i(n-m_i)>m_1(n-m_1)
+\sum\limits_{i=2}^{k}m_i(n-\sum\limits_{i=2}^{k}m_i)
$$
$$
=m_1(n-m_1)+(n-m_1)(n-(n-m_1))\geqslant 2(n-1).
$$
Hence,~(\ref{eq4.6}) does not hold if $k\geqslant 3$. Thus, $k=2$, i.e.
$G$ is a complete bipartite graph. Let us consider bipartition~(\ref{eq4.7*}). Then $n=m_1+m_2$ and
$$
\card \dip(X)=m_1(n-m_1)+m_2(n-m_2)=2m_1(n-m_1).
$$
Using~(\ref{eq4.6}) we obtain
\begin{equation}\label{eq4.10}
2(n-m_1)m_1=2(n-1).
\end{equation}
Quadratic equation~(\ref{eq4.10}) has the roots $1$ and $n-1$. Thus, $\card X_1=1$ or
$\card X_2=1$ as required.
\end{proof}

Theorem~\ref{th4.2} can be reformulated in the graph theory language.
\begin{corollary}\label{cor4.5}
Let $G=(V,E)$ be a finite complete $k$-partite graph with $k \geqslant 2$.
Then
\begin{equation}\label{eq4.12}
\card E \geqslant \card V-1.
\end{equation}
Equality in~(\ref{eq4.12}) is attained if and only if $G$ is a star.
\end{corollary}

\textbf{Acknowledgment.} The second author is thankful to the Finnish
Academy of Science and Letters for the support.

{\bf D. Dordovskyi}

Institute of Applied Mathematics and Mechanics of NASU, R. Luxemburg str. 74, Donetsk 83114, Ukraine

{\bf E-mail: } dordovskydmitry@gmail.com
\bigskip

{\bf O. Dovgoshey}

Institute of Applied Mathematics and Mechanics of NASU, R. Luxemburg str. 74, Donetsk 83114, Ukraine

{\bf E-mail: } aleksdov@mail.ru
\bigskip

{\bf E. Petrov}

Institute of Applied Mathematics and Mechanics of NASU, R. Luxemburg str. 74, Donetsk 83114, Ukraine

{\bf E-mail: } eugeniy.petrov@gmail.com

\end{document}